\documentclass[reqno,10pt]{amsart}
\usepackage{amscd,amssymb,amsmath,color}
\usepackage{latexsym}
\usepackage{verbatim}
\usepackage[all]{xy}
\usepackage[colorlinks=true]{hyperref}
\usepackage{defs}

\begin{document}

\author{Uri Brezner}
\address{Einstein Institute of Mathematics, The Hebrew University of Jerusalem, Giv'at Ram, Jerusalem, 91904, Israel}
\email{uri.brezner@mail.huji.ac.il}


\title{Tropical reduction and lifting of $q$-differentials on Berkovich curves}

\begin{abstract}
Given a complete real-valued field $k$ of residue characteristic zero, we study properties of a meromorphic $q$-differential form $\xi$ (a section of $\Omega_X^{\otimes q}$) on a smooth proper $k$-analytic curve $X$. In particular, we associate to $(X,\xi)$ a natural tropical reduction datum combining tropical-geometric data over the value group of $|k^\times|$ and algebro-geometric reduction data over the residue field $\widetilde{k}$. We show that this datum satisfies natural compatibility conditions, and prove a lifting theorem asserting that any compatible tropical reduction datum lifts to an actual pair $(X, \xi)$.
This generalizes the result of \cite{TT20} from $q=1$ to a general natural number $q$.
Furthermore, it is a non-Archimedean analog of \cite{BCGGM16}.
\end{abstract}
\maketitle

\section{Introduction}
\subsection{Background and Goal}
\subsubsection{Complex curves with differential forms}
In \cite{BCGGM18} Bainbridge, Chen, Gendron, Grushevsky and M\"{o}ller study the geometry of the moduli space of smooth complex curves of genus $g$ equipped with a meromorphic differential form (a section of the canonical line bundle perhaps twisted) $\Omega\calM_g$.
Toward that end they provide the moduli space with a stratification by the pattern of poles and zeros and provide a complete description for the closures of the strata in a certain compactification. More precisely, the stratification is given by a tuple of integers $\bfm=(m_1\.m_n)$ such that $\sum m_i=2g-2$, and the stratum $\Omega\calM_g(\bfm)$ parameterizes $(X, \omega, z_1\. z_n)$ where $X$ is a smooth curve of genus $g$, $z_1\. z_n\in X$ are distinct marked points, and $\omega$ is a differential form such that $\div(\omega)=\sum m_iz_i$.
The authors then describe the boundary points of a stratum in terms of a finite graph $\Gamma$, a real valued function defined on the set of vertices $\ell\colon V(\Gamma)\to \bbR$ and a set of differential forms $\{\omega_v\}$, one for each vertex of $\Gamma$ satisfying certain combinatorial and residue compatibility conditions. Probably, the main surprise was the discovery of a rather involved and not so intuitive \emph{global residue condition} (\cite[Definition~1.2(4)]{BCGGM18}), which had no analogs in previous works on adjacent fields.

\subsubsection{Complex curves with $q$-differential forms}
In a follow up work \cite{BCGGM16} the authors extend their methods to study the geometry of $\Omega^q\calM_g$ that is the moduli space of smooth complex curves of genus $g$ equipped with a meromorphic $q$-differential form (a section of the $q$-th power of the canonical line bundle twisted by the pole divisor).
This study leads the authors to the notion of the \emph{q-residue} of a $q$-differential form at a pole (\cite[\S~3.1]{BCGGM16}) and an even more involved \emph{global q-residue condition} (\cite[Definition ~1.4.(4)]{BCGGM16}).

\subsubsection{Berkovich curves with differential forms}
In \cite{TT20} Temkin and Tyomkin re-interpret the main parts of \cite{BCGGM18} in terms of non-archimedean geometry.
In this view the graph $\Gamma$, the function $\ell$ and set of differential forms $\{\omega_v\}$ associated to a $k$-analytic curve $X$ equipped with a meromorphic differential form $\omega$ are obtained as a tropical reduction datum of the pair $(X,\omega)$ (see \cite[Section~1.2]{TT20} for a short discussion on the meaning of tropical reduction).
The result is a much more algebraic approach that applies over any complete valuation field $k$ of residual characteristic zero.
More importantly, by introducing a new invariant, the \emph{residue function} (taking values in the base field $k$ and not its residue field!) that satisfies some natural compatibility conditions, the authors are able to recover the global residue condition of \cite{BCGGM18} and obtain a short proof of main result of \cite{BCGGM18}.

\subsubsection{Motivation}
The motivation for this work is to extend the results of \cite{TT20} to study the structure of $q$-differential forms on $k$-analytic curves.
In particular, we will re-interpret the main parts of \cite{BCGGM16} and obtain a short proof of  \cite{BCGGM16} Theorem 1.5. To achieve this, given a nice $k$-analytic curve $X$ (see \S\ref{sec notation}) and a meromorphic $q$-differential form $\xi$ we construct a tropical reduction datum of $(X, \xi)$, describe a list of compatibilities this datum satisfies, and show that the list is exhaustive by proving a lifting theorem - any compatible tropical reduction
datum lifts to an actual pair $(X, \xi)$.

\subsubsection{$q$-differentials}
A differential form $\omega$ defined on a complex curve $X$ induces a flat metric with conical singularities on $X$.
In this context pairs $(X,\omega)$ are called \emph{translation surfaces} and they arise in the study of various dynamical systems and in Teichm\"{u}ller theory.
Closely related are \emph{half-translation surfaces} consisting of a complex curve $X$ and a quadratic differential form (that is a $2$-differential form) $\eta$ on $X$ as $\eta$ also induces a flat metric with singularities on $X$ (surveys \cite{W}, \cite{M18}).
In light of this it is not surprising that there has been a rising interest in the more general $q$-differential forms in recent years, for example \cite{BCGGM16}, \cite{S} \cite{GT}.

\subsection{Structure of the paper}
Section \ref{sec trop_red} is dedicated to the definition and study of a tropical reduction datum of a pair $(X, \xi)$.
The lifting theorem is started and proved is Section \ref{sec lift}.
The proof consists of two parts: construction of local liftings and a patching argument.
The main ingredients of the first part are the canonical cover construction for algebraic curves of \cite{BCGGM16} and an equivariant analogue of the local lifting of \cite{TT20}.
For the patching argument the main tool is the existence of a good coordinate of an annulus with respect to a $q$-differential form.
Section \ref{sec good coor} is dedicated to the definition and proof of existence of a good coordinate.

\subsection{Notation and convention} \label{sec notation}
We follow the notation and conventions of \cite{TT20}.
Throughout $k$ denotes a fixed algebraically closed complete non-archimedean field, its valuation is written in additive notation as $\nu \colon  k \to  \bbR\cup \{\infty\}$ and in multiplicative notation as $|\phantom{m}| \colon  k \to  \bbR$ (with $|c|= 10^{-\nu(c)}$), $\kcirc$ is the ring of integers, $\kcirccirc$ its maximal ideal, and $\tilk=\kcirc/\kcirccirc$ is the residue field.
It is assumed that $\chark=\cha(\tilk)=0$.

By a \emph{nice} $k$-analytic curve we mean a quasi-smooth, connected, compact, separated, strictly $k$-analytic curve.
By a \emph{star-shaped} curve we mean a pair $(X, x)$ where $X$ is a nice $k$-analytic curve, $x \in X$ a point of type 2, and $X \setminus \{x\}$ a disjoint union of open discs and semi-open annuli.

Let $q$ be a natural number.
Abusing language, by a \emph{q-differential form} or a \emph{q-differential} on a curve $X$ we mean a section of $\left(\Omega^1_X\right)^{\otimes q}$ the $q$-th power of the sheaf of differential form of $X$.

A branch of $X$ at $x$ is an equivalence class of germs of intervals $[x, y] \subset X$ and the set of all branches is denoted $\Br(x)$.
In particular, there is one branch at a point of type 1, and for a point $y$ of type 2 there is a one-to-one correspondence between the branches $e \in \Br(y)$ and the $\tilk$-points $p_e$ of the reduction curve $C_y$.

The tropical curves are skeletons of analytic curves with divisors in the sense of \cite{TT20}.
We denote the set of vertices of type $i$ by $V_i(\Gamma)$ so $V(\Gamma)=V_1(\Gamma)\cup V_2(\Gamma)$ is the set of all vertices.
The edges adjacent to the vertices of type 1 are called \emph{legs}, and have infinite length.
Other edges are \emph{bounded}.
All edges of tropical curves are \emph{oriented}, and each bounded edge (even a loop) is considered twice by equipping it with the two possible orientations.
The legs are always oriented towards the vertex of type 1.
By abuse of notation the set of oriented edges and legs is denoted by $E(\Gamma)$, and the set of legs by $L(\Gamma)$.
If $e \in E(\Gamma)$ is bounded then $e^\opp$ denotes the same edge with the opposite orientation.

The tail and the head functions of oriented edges are denoted as $\gtt\colon  E(\Gamma) \to V (\Gamma)$ and $\gth\colon E(\Gamma) \to V (\Gamma)$ respectively.
If $\Gamma$ is a tropical curve and $x \in V_2(\Gamma)$ then $\Star(x)$ denotes the set of oriented edges $e\in E(\Gamma)$ such that $\gtt(e)=x$.
We shall not distinguish between tropical curves and their geometric realizations as metric graphs.

Given $0<r<1$ we denote by $\calA[r,1]$ (resp.\ $\calA(r,1)$) the closed (resp.\ open) annulus of modulus $0 < r < 1$ whose skeleton is the segment $[r,1]$ (resp.\ $(r,1)$).

\section{Tropical Reduction}  \label{sec trop_red}
\subsection{Outline}
Let $X$ be a nice $k$-analytic curve equipped with a non-zero meromorphic $q$-differential form $\xi$ and a compatible skeleton $\Gamma$ (see below).
Generalizing \cite{TT20}, we will associate to $(X,\Gamma,\xi)$ a tropical reduction $(\calC, \tilxi^{\gr},\gtR_{\xi}^q)$ consisting of a metrized curve complex with boundary $\calC =(\Gamma,\{C_x\}_{x\in V_2(\Gamma)},\{p_e\}_{e \in E (\Gamma) } ) $,  a collection of graded elements $\tilxi^{\gr}_x \in \Omega_{\tilk(C_x)/\tilk}^{\otimes q} \otimes_{\tilk} k^{\gr}$ for $x \in V_2(\Gamma)$, and a function $\gtR^q_{\xi}\colon  E(\Gamma) \to k$ satisfying certain compatibility conditions.
The main novelty is the $q$-residue function $\gtR_{\xi}^q$ we will define below.
Note that \cite{TT20} studies the case of $q=1$.

\subsection{Skeleton of a pair}
Given $X$ and $\xi$ as above,
by a \emph{skeleton of $(X,\xi)$} we mean a skeleton $\Gamma$ of $X$ containing all zeros and poles of $\xi$.
 For the remainder of this section we fix a triple $(X,\Gamma,\xi)$.

\subsection{The metrized curve complex}
Exactly as in \cite{TT20}, we enrich the skeleton $\Gamma$ to a metrized curve complex with boundary by $\calC$ by attaching to each type 2 vertex $x$ its reduction $\tilk$-curve $C_x$, and for each edge $e\in\Star(x)$ marking the corresponding closed point $p_e\in C_x$.
For more details see \cite[\S~2.2]{TT20}.

\subsection{The level function}
Let $||\phantom {m}||$ be the K\"ahler norm on $\Omega_X^{\otimes q}$ as defined in \cite[\S~8]{Metrization}.
The induced level function $\ell_\xi\colon \Gamma \to \bbR$ is given by $\ell_\xi(x)=-\log_{10} ||\xi||_x$.
As in \cite{TT20}, the level function $\ell_\xi$ is piecewise $\nu(k^\times)$-integral, affine and continuous (by \cite[Theorem~8.1.6]{Metrization}).

\subsubsection{Graded reduction, scaled reduction and level function}\label{sec graded and scaled red}
Given $x \in X$ we can associate two reductions to a section $\eta\in \Omega_{X,x}^{\otimes q}$.
Choosing $c\in k$ such that $|c|=||\eta||_x$ the (usual) reduction of $c^{-1}\eta$ lies in
 $\left(\Omega_{X,x}^{\otimes q}\right)^{\circ}/\left(\Omega_{X,x}^{\otimes q}\right)^{\circcirc}=\Omega_{\tilk(C_x)/ \tilk}^{\otimes q} $. 
Choosing a different scaling factor $c'$ the resulting reduction differs by an element of $k^\times$.
The \emph{scaled reduction} of $\eta$ at the point $x$ is the well-defined class $\tileta_x\in \left(\Omega_{\tilk(C_x)/ \tilk}\right)^{\otimes q}/\tilk^\times$ of forms modulo scaling.

On the other hand, denoting by $\left(\Omega_{X,x}^{\otimes q}\right)^{\le r}$ the subspace of sections $\eta\in \Omega_{X,x}^{\otimes q}$ for which $||\eta||_x \le r$,
then $\left(\Omega_{X,x}^{\otimes q}\right)^{\le r}$ defines a filtration on $\Omega_{X,x}^{\otimes q}$ and we denote
the associated graded $\tilk(C_x)$-module by $\left(\Omega_{x}^{\otimes q}\right)^\gr$.
Furthermore, each $0\neq c \in k$ induces an isomorphism between the $|c|$-graded part of $\left(\Omega_{x}^{\otimes q}\right)^\gr$ and $\left(\Omega_{X,x}^{\circ}\right)^{\otimes q}/\left(\Omega_{X,x}^{\circ\circ}\right)^{\otimes q}=\Omega_{\tilk(C)/ \tilk}^{\otimes q}$, and in fact, $\left(\Omega_{x}^{\otimes q}\right)^\gr$ is canonically isomorphic to $\Omega_{\tilk(C_x)/\tilk}^{\otimes q} \otimes_{\tilk} k^{\gr}$.
Each section $\eta\in \Omega_{X,x}^{\otimes q}$ induces a uniquely defined graded element $\tileta^\gr_x\in \left(\Omega_{x}^{\otimes q}\right)^\gr$ called the \emph{graded reduction} of $\eta$ at the point $x$.
The graded reduction $\tileta^\gr_x$ is in fact a representative of the scaled reduction $\tileta_x$ placed in the $||\eta||_x=10^{\ell_\eta(x)}$-graded part of $\left(\Omega_{x}^{\otimes q}\right)^\gr$.
Thus, the datum $\tileta^\gr_x$ is equivalent to the datum $(\tileta_x, \ell_\eta(x))$ and we switch between them freely.

\subsubsection{Slope}
For a point $x \in V_2(\Gamma)$ and $e\in \Star(x)$ we denote the slope of the level function $\ell_\xi$ along $e$ by $\frac{\partial \ell_\xi}{\partial e}$.

\subsubsection{Order}
Let $C$ be a $\tilk$-curve.
For $\eta=f(z)\cdot(dz)^q \in \Omega_{\tilk(C)/ \tilk}^{\otimes q}$ we denote
\[\ord_{p}(\eta)=\ord_{p} (f)\]
where $\ord_{p} (f)$ is the usual order of a function at a closed point $p \in C$ and $z=0$ at the point $p$.

\begin{prop} \label{prop slope}
For any $x \in V_2(\Gamma)$ and $e\in \Star(x)$ we have
\[ \frac{\partial \ell_\xi}{\partial e} =-\ord_{p_e}(\tilxi_x^\gr) -q .\]
\end{prop}
\begin{proof}
Choose $c\in k$ such that $|c^{-1}|=||\xi||_x$ and set $\eta=\wt{c\xi}$ with reduction taken at $x$.
By \cite[Lemma~3.3.2]{BT20} $-\frac{\partial \ell_\xi}{\partial e}$ equals to the induced order of $\eta$ in the reduction of $\left(\Omega_{X_G,C_x}\right)^{\otimes q}$ at $x$ which is $\left(\Omega_{\tilk(C_x)/\tilk}^\log\right)^{\otimes q}$ by \cite[Theorem~3.3.4]{BT20} and the fact that tensor products commute with colimts.
As the induced order on $\Omega_{\tilk(C_x)/\tilk}^\log$ assigns $\ord_{p_e}(f)+1$ to the form $fdt$,
it follows that the induced order on $\left(\Omega_{\tilk(C_x)/\tilk}^\log\right)^{\otimes q}$ of $\eta$ is given by $\ord_{p_e}(\eta) +q$.
\end{proof}

\begin{cor}
For any $x \in V_2(\Gamma)$
\[\div(\tilxi^\gr_x)=\sum_{e\in \Star(x)} \left(-\frac{\partial \ell_\xi}{\partial e}-q\right)p_e\]
and if $x$ is not a boundary point then
\[ \sum_{e\in \Star(x)}\frac{\partial \ell_\xi}{\partial e}= q\left( 2-2g(C_x)-|\Star(x)| \right) . \]
\end{cor}

\subsection{The $q$-residue function}
As in \cite{TT20} the last part of our tropical reduction datum is the \emph{q-residue function} $\gtR_\xi ^q\colon E(\Gamma) \to k$.
The definition follows the same steps as that of the residue function $\gtR_\omega$ in \cite{TT20}, that is we first define a $q$-residue $\Res^q_\calA(\xi)\in k$ of a $q$-differential form $\xi$ along an oriented annulus $\calA$.
Then we extend the definition of a $q$-residue to a $q$-differential form $\xi$ along a branch $e$ of a nice curve $X$, denoted $\Res^q_e(\xi)$.
Finally we define the $q$-residue function $\gtR_\xi ^q\colon E(\Gamma) \to k$ of $(X,\Gamma,\xi)$.
\begin{rem}
Recall that the \emph{induced orientation} of an annulus $\calA$ with skeleton $e$ is the one along which the coordinate is decreasing.
\end{rem}
\subsubsection{The $q$-residue along an annulus}
Let $\xi$ be a $q$-differential form without zeros and poles on an oriented analytic annulus $\calA$.
If there exists a differential form $\omega$ on $\calA$ such that $\xi=\omega^q$ then we set $\Res^q_\calA(\xi)$ to be the $q$-th power of $\Res_\calA(\omega)$ (defined in \cite[\S~2.4.1]{TT20}).
Else, we set  $\Res^q_\calA(\xi)=0$.

Note that $\Res^q_\calA(\xi)$ is well defined since if such an $\omega$ exists, then any other differential form $\omega'$ satisfying $\xi={\omega'}^q$ differs form $\omega$ by a multiple of a $q$-th root of unity.
It follows that  $\Res_\calA(\omega')$ differs form $\Res_\calA(\omega)$ by a multiple of a $q$-th root of unity as well so $\left(\Res_\calA(\omega')\right)^q=\left(\Res_\calA(\omega)\right)^q$.

\begin{rem}
Recall that for a differential form $\omega$ without zeros and poles on an oriented annulus $\calA$ with coordinate $t$, then $c=\Res_\calA(\omega)$ is the unique scalar for which the differential form $\omega-c\frac{dt}{t}$ is exact (see \cite[\S~2.4.1]{TT20}).
\end{rem}

\begin{lem} \label{lem q-form is q power annulus}
Let $\xi$ be a $q$-differential form $\xi$ with neither poles nor zeros on an annulus $\calA$.
Set $\ell(x)=\log ||\xi||_x$.
Then $\xi$ is a $q$ power of some differential form $\omega$ if and only if $q$ divides $\frac{\partial \ell_\xi}{\partial e}$ where $e$ is the skeleton of $\calA$.
\end{lem}

\begin{proof}
It is enough to show for a closed annulus.
Choosing a coordinate $s$ of $\calA$ we can write $\xi=\sum_i a_is^i (\frac{ds}{s})^q$.
Since $\xi$ has neither poles nor zeros it has a dominant term, i.e.\ there exists $n\in \bbZ$ such that $|a_is^i|_x< |a_ns^n|_x$ for all $i\neq n$ and all $x\in \calA$ (see \cite[\S~4.1.2]{BT20}).
We write
\[\xi =a_ns^n\left( 1+\sum_{i\neq n}\frac{a_i}{a_n}s^{i-n} \right)\left( \frac{ds}{s} \right)^q.\]
Note that $1+\sum_{i\neq n}\frac{a_i}{a_n}s^{i-n}$ is a $q$ power.
It follows that $\xi$ is a $q$ power of some differential form $\omega$ if and only if $q$ divides $n=-\frac{\partial \ell_\xi}{\partial e}$ (the minus sign is due to the induced orientation).
\end{proof}

\subsubsection{The $q$-residue along a branch}
Given $(X,\xi)$, for a point $x$ of type 2 and $e\in \Br(x)$ with $\gtt(e)=x$ there is an open annulus $\calA \subset X$ such that its skeleton $e_\calA$ lies along $e$ and $\xi|_\calA$ has neither zeros nor poles on the interior of $e_\calA$.
We define the $q$-residue of $\xi$ along $e$ to be the $q$-residue of $\xi|_\calA$ along  $\calA$ and write $\Res^q_e(\xi)=\Res^q_\calA(\xi|_\calA)$.

\subsubsection{The $q$-residue function $\gtR^q_\xi$}
Finally, for $(X,\Gamma,\xi)$ for each edge $e\in E(\Gamma)$ we set $\gtR^q_\xi(e)=\Res^q_e(\xi)$.

\begin{prop} \label{prop_properties of residue function}
Given $(X,\Gamma,\xi)$ as above
\begin{enumerate}
	\item $\gtR_\xi^q(e)=(-1)^q\gtR_\xi^q(e^\opp)$ for any bounded edge $e$.
	\item $\gtR_\xi^q(l)=\res_{x_l}^q (\xi)$ for any leg $l$ adjacent to a point $x_l$ of type 1.
	\item $\wt{\gtR_\xi^q(e)}^\gr=\res_{p_e}^q (\wt{\xi}_x^\gr)$ for any vertex $x \in \Gamma$ of type 2 and any $e \in \Star(x)$.
\end{enumerate}
Where $\res_{p_e}^q$ in (3) denotes the $q$-residue at the point $p_e\in C_x$ in the sense of \cite[\S~3.1]{BCGGM16}.
\end{prop}
\begin{proof}
The claims follow from the definition and \cite[Proposition~4.2.5]{TT20}.
\end{proof}

\subsection{$q$-harmonicity}

    First we recall the definition of $P_{s,d}$ from \cite[(1.1)]{BCGGM16}
\begin{defin}
For natural numbers $s,d$ and $R_1\. R_s\in k$ set
\[P_{s,d}(R_1\. R_s)=\prod_{\{(r_1\. r_s)\in k^s \mid r_i^d=R_i\}} \sum_{i=1}^s r_i .\]
\end{defin}
\begin{rem}
(1) As $P_{s,d}$ is symmetric with respect to the $d$-th roots of the $R_i$'s it depends only on the $R_i$'s.

(2) $P_{s,d}=0$ if and only if $ \sum_{i=1}^s r_i =0$ for some $s$-tuple $(r_1\. r_s)$ such that $r_i^q=R_i$ for all $i$.
\end{rem}

\begin{prop} \label{prop q-harmonicity}
Let $x \in V_2(\Gamma)$ be a point not in the boundary with $\Star(x)=\{e_1\. e_n\}$.
If $\xi$ is a $q$-power of a differential form $\omega$ on a neighbourhood of $x$ then
\[ P_{n,q}(\gtR^q_\xi(e_1)\. \gtR^q_\xi(e_n))=0.\]
\end{prop}

\begin{proof}
By construction, $\gtR^q_\xi(e)=\left(\gtR_\omega(e)\right)^q=\left(\Res_e(\omega)\right)^q$ for every $e\in \Star(x)$.
The claim follows from (2) of the Remark above and \cite[Theorem~3.1.1]{TT20}.
\end{proof}

\begin{rem}\label{rem q-harmonicity conditions}
For a point $x \in V_2(\Gamma)$ not in the boundary the assumption of the last Proposition is satisfied exactly when $q$ divides $\frac{\partial \ell_\xi}{\partial e}$ for all $e\in \Star(x)$
by Lemma \ref{lem q-form is q power annulus}.

Furthermore, denoting $\Star(x)=\{e_1\.e_n\}$ and $p = \{p_1\.p_n\}\subset C_x$ the closed points corresponding to $\Star(x)$
we have $\div(\xi)=\sum_{i=1}^nm_ip_i$ for some $m_i\in\bbZ$.
Set
\[ \tag{$\star$}\begin{aligned}
l_i=&\gcd(q,m_i)   &  l&=\gcd(l_1\.l_n) \\
d_i&=\frac{q}{l_i}  &  d=&\frac{q}{l}=\lcm(d_1\.d_n).
\end{aligned} \]
Then $\xi$ is a $q$-power of a differential form $\omega$ on a neighbourhood of $x$ if and only if all $m_i$ are divisible by $q$ that is $d=1$.
\end{rem}

\subsection{The tropical reduction datum}
Let $\gamma = \left( \calC, \tileta^\gr,\gtR^q\right)$ be a triple consisting of a metrized curve complex with boundary
$\calC = \left(\Gamma, (c_x)_{x\in V_2(\Gamma)},(p_e)_{e\in E(\Gamma)}\right)$,
a collection of graded elements $\tileta^\gr_x\in \left( \Omega_{\tilk(C_x)/\tilk}\right)^{\otimes q} \otimes_\tilk \tilk^\gr$
and a function $\gtR^q\colon E(\Gamma) \to k$.
Like in \cite{TT20}, there is a unique continuous, piecewise affine function $\ell_\gamma \colon \Gamma \to \bbR$ such that (i) $\ell_\gamma$ is affine on the edges of $\Gamma$, (ii) $10^{\ell_\gamma(x)}$ is the grading of $\tileta^\gr_x$ for all $x\in V_2(\Gamma)$, and (iii) for any leg $l\in L(\Gamma)$, the slope of $\ell_\gamma$ along the leg is equal to $-q-\ord_{p_l}(\tileta^\gr_x)$ for $x=\gtt(l)$ the vertex in $V_2(\Gamma)$ adjacent to $l$.

\begin{defin} \label{def red datum}
A triple $\gamma$ is called a \emph{tropical reduction datum} if the following compatibilities hold:
\begin{enumerate}
	\item $\frac{\partial \ell_\gamma}{\partial e}=-q-\ord_{p_e}(\tileta^\gr_x)$ for any $e\in E(\Gamma)$ with $\gtt(e)=x$,
	\item $\wt{\gtR(e)}^\gr=\res_{p_e}^q (\tileta_x^\gr)$ for any $e\in E(\Gamma)$ with $\gtt(e)=x$,
	\item if $x\in V_2(\Gamma)$ is not in the boundary with $\Star(x)=\{e_1\. e_n\}$ and $q$ divides every $\frac{\partial \ell_\gamma}{\partial e_i}$ then $ P_{n,q}(\gtR^q(e_1)\. \gtR^q(e_n))=0$,
	\item if $q$ does not divide $\frac{\partial \ell_\gamma}{\partial e}$ for $e\in E(\Gamma)$ then $\gtR^q(e)=0$,
	\item  if $\frac{\partial \ell_\gamma}{\partial l}<0$ for $l\in L(\Gamma)$ then $\gtR^q(l)=0$.
\end{enumerate}
\end{defin}

\begin{theor}
Let $\left(X,\Gamma,\xi\right)$ be a nice $k$-analytic curve equipped with a non-zero meromorphic $q$-differential form $\xi$ and a compatible skeleton $\Gamma$ (i.e.\  $\div(\xi)\subset \Gamma $). Let $\gamma \left(X,\Gamma,\xi\right)$ be the triple consisting of the metrized
curve complex with boundary $\calC$ associated to $\left(X, \Gamma\right)$, the graded reductions $\tilxi_x^\gr$ for each $x\in V_2(\Gamma)$, and the $q$-residue function $\gtR^q_\xi$. Then $\gamma \left(X,\Gamma,\xi\right)$ is a tropical reduction datum.
\end{theor}

\begin{proof}
This is the content of Propositions \ref{prop slope}, \ref{prop_properties of residue function} and \ref{prop q-harmonicity}.
\end{proof}

\section{Lifting a tropical reduction datum} \label{sec lift}
\subsection{Star-shaped objects} \label{sec star-shape}
Recall that a star-shaped curve is a pair $(X, x)$ where $X$ is a nice $k$-analytic curve and $x \in X$ is a point of type 2 such that $X \setminus \{x\}$ is a disjoint union of open discs and semi-open annuli.
Similarly, by a \emph{star-shaped tropical curve} we mean a pair $(\Gamma, x)$ where $\Gamma$ is a tropical curve and $x$ a point of type 2 such that $\Gamma \setminus \{x\}$ is a disjoint union of open and semi-open line segments.
By a \emph{star-shaped tropical reduction datum} $(\gamma,x)$ or \emph{star-shaped metrized curve complex} $(\calC,x)$ we mean a tropical reduction datum $\gamma$ or metrized curve complex $\calC$ with underlying tropical curve $Gamma$ and $x\in V_2(\Gamma)$ such that $(\Gamma,x)$ is star-shaped.

\subsection{The main theorem}
The remainder of this section is dedicated to proving the main result of this paper
\begin{theor} \label{theor lifting}
For any tropical reduction datum $\gamma = \left( \calC, \tilxi^\gr,\gtR^q\right)$, there exists a nice $k$-analytic curve $X$ equipped with a non-zero meromorphic $q$-differential form $\xi$ and a compatible skeleton $\Gamma$ such that the tropical reduction of $(X, \Gamma, \xi)$ is $\gamma$.
\end{theor}

Our strategy of proof is as follows.
We first prove the theorem for a star-shaped case.
That is, given a star-shaped tropical reduction datum $(\gamma,x)$ we construct a star-shaped curve $(X,x)$ with a star-shaped skeleton $(\Gamma,x)$ and a non-zero meromorphic $q$-differential form $\xi$ such that the tropical reduction of $(X, x, \Gamma , \xi)$ is the given star-shaped tropical reduction datum.
In order to achieve this we will use the canonical cover construction of \cite[\S~2.1]{BCGGM16} and the lifting theorem of \cite{TT20} to obtain a "cover" of the required star-shaped curve.
More precisely, we will construct a star-shaped curve and skeleton with a non-zero meromorphic ($1$-)differential form $(\hatX,\hatx, \hatGamma, \hatomega)$ equipped with a group action of $\mu_d$ for an appropriate $d$.
The quotient of $(\hatX,\hatx, \hatGamma, \hatomega^q)$ by the group action is the desired star-shaped lift.

For the general case, given a tropical reduction datum $\gamma = \left( \calC, \tilxi^\gr,\gtR^q\right)$ we break it into star-shaped tropical reduction data.
Then we patch together all the star-shaped curves and $q$-differential forms to obtain the desired $(X, \Gamma, \xi)$.

\subsection{Lifting an automorphism}
We first need to slightly extend the result \cite[Lemma~4.3.2]{TT20} to include also an automorphism.
\begin{lem} \label{lem lift of auto}
Let $C$ be a smooth proper curve over $\tilk$ and $p = \{p_1\.p_n\}\subset C$ a finite set of closed points.
Assume that $C$ is equipped with an automorphism $\tiltau\colon C \to C$ of order $d$ such that $p$ is $\tiltau$-invariant.
Then there exists a nice proper $k$-curve $Y$ with reduction $C$, a set of $k$-points $P = \{P_1\.P_n\}$ lifting $p$, an automorphism $\tau\colon Y \to Y$ of order $d$ that lifts $\tiltau$ and $P$ is $\tau$-invariant.
\end{lem}

\begin{proof}
Choosing a section $\tilk \hookrightarrow \kcirc$ of the reduction homomorphism $ \kcirc \to \tilk$ we take our lifts $Y$, $P$ and $\tau$ to be the analytifications of $C\otimes_{\tilk}k$, $p\otimes_{\tilk}k$, and $C\otimes_{\tilk}k\to C\otimes_{\tilk}k$.
\end{proof}

\subsection{Lifting a star-shaped tropical reduction datum}
Fix a star-shaped tropical reduction datum $(\gamma,x) = \left( \calC, x, \tilxi^\gr,\gtR^q\right)$.

\subsubsection{Notation} \label{sec notation lift}
Set $\Star(x)=\{e_1\.e_n\}$ and $p = \{p_1\.p_n\}\subset C_x$ the closed points corresponding to $\Star(x)$.
We have $\div(\tilxi^\gr_x)=\sum_{i=1}^nm_ip_i$ for some $m_i\in\bbZ$.
We denote $l_i,l,d_i,d$ as in ($\star$).
Also set $R_i=\gtR^q(e_i)$.

\subsubsection{Canonical cover of an algebraic curve}
In the course of the poof of Theorem \ref{theor lifting} we will need to distinguish between two cases - whether or not the $q$-differential form is locally a $q$-power of some differential form.
The significance of the canonical cover construction of \cite{BCGGM16} is that it proves that for any $q$-differential form $\eta$ on a curve $C$ there exists a suitable cover $\hatC \to C$ such that the pullback of $\eta$ to $\hatC$ is in fact a $q$-power of some differential form.
We briefly recall the main properties of this construction.

Let $\eta$ be a $d$-differential form on an algebraic, integral, smooth $\tilk$-curve $C$.
Assume that $\eta$ is not a power of any $d'$-differential form on $C$ where $d'<d$.
Write $\div(\eta)=\sum_{i=1}^nm_ip_i$.
By the canonical cover construction \cite[\S~2.1]{BCGGM16} we obtain a flat cyclic $d$-cover $\pi\colon \hatC\to C$, a meromorphic differential form $\omega$ on $\hatC$ such that $\left(\omega\right)^d=\pi^*(\eta)$ and also a deck transformation $\tau$ of the cover of order $d$.
By construction , $\hatC$ is integral and normal and $\pi$ is \'etale over $C-\div(\eta)_\red$.

Furthermore
\begin{itemize}
	\item $\omega$ is unique up to multiplication with a power of $\zeta$.
	\item $\tau$ generates the deck transformation group and $\tau^*\omega=\zeta\omega$ for a primitive $d$-th root of unity $\zeta$.
	\item The map $\pi\colon \hatC\to C$ is a $\mu_d$-Galois branched cover.
	\item The branch locus  of $\pi$ is contained in $\{p_1\.p_n\}$.
	\item For each $i$ the fibre $\pi^{-1}(p_i)$ contains $\gcd(d,m_i)$ distinct points and the ramification index along $p_i$ is $\frac{d}{\gcd(d,m_i)}$.
\end{itemize}
For more details see \cite[\S~2.1]{BCGGM16} and \cite[\S~3]{EV92}.

\subsubsection{Proof of Theorem \ref{theor lifting} for the star-shaped case and $x$ not in the boundry}
We need to consider two cases: $d=1$ and $d>1$ where $d$ is defined in \S ~\ref{sec notation lift}.

Assume $d=1$.
In this case, by the $q$-harmonicity condition \ref{def red datum}(3),  we can choose $r_1\.r_n\in k$ such that $\sum_ir_i=0$ and $r_i^q=R_i$ for each $i$.
Furthermore, there exists a form $\tilomega$ on $C_x$ such that $\tilomega^q=\tilxi$ whose  residue at the point $p_i$ is $\tilr_i$ for all $i$.
Setting $\gtR(e_i)=r_i$ we obtain a function  $\gtR\colon \Gamma \to k$ that is a $q$-th root of the $q$-residue function $\gtR^q$.
Finally, placing $\tilomega$ at level $\frac{\ell(x)}{q}$ we have construced a star-shaped tropical reduction datum $\gamma'= \left( \calC, \tilomega^\gr,\gtR\right)$ in the sense of \cite{TT20} which can be lifted, according to \cite[Theorem~3.3.1]{TT20}, to a proper nice $k$-curve $\hatY$ equipped with a differential form $\omega$ lifting $\gamma'$.
Setting $\xi=\omega^q$ we obtain the desired lifting.

Assume $d>1$.
In this case $d$ is the minimal number such that there exists a $d$-form $\tileta$ on $C_x$ such that $\tilxi$ is a power of $\tileta$, specifically $\tileta^l=\tilxi$.
Then $\div(\tileta)=\sum_{i=1}^nm_i'p_i$ with $m'_i=\frac{m_i}{l}$.
By the canonical cover construction of \cite[\S~2.1]{BCGGM16} we obtain a flat cyclic $d$-cover $\pi\colon \hatC\to C_x$, a meromorphic differential form $\tilomega$ on $\hatC$ such that $\left(\tilomega\right)^d=\pi^*(\tileta)$ (hence $\left(\tilomega\right)^q=\pi^*(\tilxi)$) and a deck transformation $\tau$ of $\pi$ of order $d$.
Furthermore, for each $i$ there are
\[\gcd(d,m'_i)=\gcd(\frac{q}{l},\frac{m_i}{l})=l_i/l=d/d_i\]
points of $\hatC$ over $p_i$ and their ramification index is $d_i$.
Note that, the fibre $\pi^{-1}(p_i)$ contains exactly $d$ points if and only if $d_i=1$ if and only if $q|m_i$. 
By construction, $\tilomega$ is equivariant, that is $g(\tilomega)=\chi(g)\tilomega$, where $\chi$ is a primitive character of $G=\Aut(\hatC/C)\toisom\mu_d$ and $g\in G$.
If $\pi^{-1}(p_i)$ contains less than $d$ points then their residues are zero.
On the other hand, if $|\pi^{-1}(p_i)|=d$ then the residues at the points of $\pi^{-1}(p_i)$ are of the form $\tila_i,\zeta_d \tila_i\.\zeta_d^{d-1} \tila_i\in\tilk$ where $\zeta_d$ is a primitive $d$-th root unity and $\tila_i^d=\tilR_i$.

We now construct a star-shaped tropical reduction datum $(\hatgamma, \hatx)= \left( \hatcalC, \hatx, \tilomega^\gr,\hatgtR\right)$ in the sense of \cite{TT20}.
Recall that a metrized curve complex consists of a skeleton with reduction $\tilk$-curves attached at type 2 points and marked points on these curves.

The skeleton $\hatGamma$ consists of a type 2 vertex $z$ for each point of $\cup_i\pi^{-1}(p_i)\subset\hatC$ together with another vertex of type 2 denoted $\hatx$ and edges $E_z$ connecting each $z$ to $\hatx$ of length $\frac{\len(e_i)}{d}$ where $\len(e_i)$ denotes the length of $e_i$ in the given skeleton $\Gamma$.
Next we attach the curve $\hatC$ at $\hatx$ and mark the closed points corresponding to the edges to obtain a metrized curve complex $\hatcalC$.
Note that $\left( \hatcalC, \hatx\right)$ is a star-shaped metrized curve complex in the sense of \S~\ref{sec star-shape}.

For the graded-reduced differential form we take $\tilomega$ placed at level $\frac{\ell(x)}{q}$ (see \S~\ref{sec graded and scaled red}).
Finally, we define the function $\hatgtR\colon E(\hatGamma) \to k$ as follows.
Let $E_z=[\hatx,z]$ be an edge of $\hatGamma$.
By construction, $E_z$ corresponds to a point $\hatp\in\pi^{-1}(p_i)\subset\hatC$ for some $i$.
If $|\pi^{-1}(p_i)|<d$ we set $\hatgtR (E_z)=0$.
If $|\pi^{-1}(p_i)|=d$ we choose a lift $b_i\in k$ of $\tila_i$ from above and set $\hatgtR (E_z)=\zeta_d^jb_i$ where $j$ is chosen such that $\widetilde{\hatgtR (E_z)}=\zeta_d^j\tila_i=\res_{\hatp}(\tilomega)$.
Note that
\[\sum_{E\in\Star(\hatx)}\hatgtR (E)=\sum b_i\sum_{j=0}^{d-1}\zeta_d^j=0\]
where the first sum is over all $i$ such that $|\pi^{-1}(p_i)|=d$.
This is condition (3) of \cite[Definition~3.1.4]{TT20}.
It is easy to check that conditions (1), (2) and (4) of the definition are also satisfied, that is $\hatgamma$ is indeed a tropical reduction datum.

According to \cite[Theorem~3.3.1]{TT20} there exists a  proper nice $k$-curve $\hatY$ equipped with a differential form $\hatomega_0$ lifting $\hatgamma$.
Furthermore, we can also lift the $\tilk$-automorphism $\tau$ to a $k$-automorphism $\hattau\colon \hatY\to \hatY$ also of order $d$ such that
the set of zeros and poles of $\hatomega_0$ is $\hattau$-invariant.\footnote{We can lift $\hattau$ first by Lemma \ref{lem lift of auto} and then the construction in the proof of \cite[Theorem~3.3.1]{TT20} does not affect the result.}
It follows that $Y'$ is equipped with a $\mu_d\toisom <\hattau>$ action and the set of zeros and poles of $\hatomega_0$ is $\mu_d$-invariant.

Replacing $\hatomega_0$ with $\hatomega=\frac{1}{d}\sum_{g\in \mu_d} \chi(g)^{-1}g(\omega)$ we obtain a $\mu_d$-equivariant form without changing the reduction and the residues, which are already equivariant.
The $q$-th power of $\hatomega$ is invariant and hence descends to $Y=\hatY/\mu_d$.
Setting $\xi$ to be the $q$-differential form induced by $\hatomega^q$ we obtain the desired local lifting of our tropical reduction datum $\gamma$.

\subsubsection{The case where $x$ is in the boundry}
We enlarge the given star-shaped tropical reduction datum $(\gamma,x)$ to a star-shaped tropical reduction datum $(\gamma',x)$ without boundary.
We then apply the previous part to obtain lift $(Y',x,\xi')$ of $(\gamma',x)$.
Finally we remove the unwanted branches of $Y'$ to obtain a star-shaped curve $(Y,x)$ as required and take $\xi=\xi'|_Y$.

Let $\oC_x$ be the smooth compactification of $C_x$.
Then $\oC_x\setminus C_x$ is a finte set of closed points $\{p_{n+1}\.p_{a}\}$.
For each $n+1\le i\le a$ we attach a leg to $x\in \Gamma$ and replace $C_x$ with $\oC$ to obtain a star-shaped skeleton without boundary.
It remains to extend $\gtR^q$ to the new legs $e_{n+1}\.e_{a}$.
Again we need to distinguish two cases in order to satisfy the $q$-harmonicity condition: $d=1$ and $d>1$.

If $d>1$ we can simply set $\gtR^q(e_i)=0$ for each $n+1\le i\le a$.

If $d=1$ we need to chose the values $R_i=\gtR^q(e_i)$ for $n+1\le i\le a$ in such a way that $P_{a,q}(\gtR^q(e_1)\. \gtR^q(e_a))=0$.
For each $1\le i \le n$ choose $r_i\in k$ such that $r_i^q=R_i$.
Now take $R_{n+1}=\left(-\displaystyle\sum_{j=1}^n r_j\right)^q$ and $R_{i}=0$ for each $n+1< i\le a$.

\subsection{Lifting a general tropical reduction datum}
We now prove Theorem \ref{theor lifting} in the general case.
 
Given a tropical reduction datum $\gamma = \left( \calC, \tilxi^\gr,\gtR^q\right)$ we cut it into star-shaped tropical reduction data, one for each type 2 vertex of the given skeleton and lift each one to a star-shaped curve as described above.

\subsubsection{Patching}
Finally, we patch together the star-shaped curves and $q$-differentials constructed in the previous section to obtain a nice $k$-analytic curve $X$ and $q$-differential $\xi$ liftting $\gamma$.
We will need the following
\begin{prop} \label{prop good coor}
If a $q$-differential $\xi$ has neither zeros nor poles on an analytic annulus $\calA$, then $\calA$ admits a coordinate $t$ such that
either of the following cases holds
\begin{enumerate}
	\item $\xi=(c_nt^{\frac{n}{q}}+c_0)^q(\frac{dt}{t})^q$ with $q|n\neq 0$ and $|c_nt^{\frac{n}{q}}|_x>|c_0|_x$ for all $x\in \calA$.
	\item $\xi=c^q_nt^n(\frac{dt}{t})^q$.
\end{enumerate}
\end{prop}
Note that in case $(1)$ we have $\Res^q_\calA(\xi)=c^q_0$ and in case $(2)$ we have $\Res^q_\calA(\xi)=c^q_0$ if $n=0$ and $\Res^q_\calA(\xi)=0$ if $q$ does not divide $n$ (if $n\neq 0$ and $q$ divides $n$ we are back in case $(1)$).
We call such a coordinate \emph{good} with respect to $\xi$.
For expositional purposes we defer the proof to \S~\ref{sec good coor}.

Let $e$ be an edge of the skeleton $\Gamma$.
We distinguish between two types of patchings.

First, assume that $e = [x, z]\in L(\Gamma)$ is a leg.
Let $(Y,x,\xi_x)$ be a star-shaped curve and $q$-differential as constructed above.
There is an open annulus $\calA$ in $Y\setminus\{x\}$ with skeleton lying along $e$. Equip $\calA$ with the $q$-differential $\xi_\calA=\xi_x|_\calA$.
By Proposition \ref{prop good coor} there exists a coordinate $t$ on $\calA$ such that either $\xi_\calA=(c_nt^{\frac{n}{q}}+c_0)^q(\frac{dt}{t})^q$ with $n\neq 0$ a multiple of $q$ or $\xi_\calA=c^q_nt^n(\frac{dt}{t})^q$, such that $n=-\frac{\partial \ell}{\partial e}$ and $c^q_0=\gtR^q(e)$.
Note that both $\frac{\partial \ell}{\partial e}$ and $\gtR^q(e)$ are determined by the given tropical reduction datum.

Assume that $\xi_\calA=(c_nt^{\frac{n}{q}}+c_0)^q(\frac{dt}{t})^q$.
Consider an open unit disc $D_z$ with origin at $z$ and a coordinate $s$.
Equip $D_z$ with the $q$-differential $\xi_z=(c_ns^{\frac{n}{q}}+c_0)^q(\frac{ds}{s})^q$, 
and glue $Y$ and $D_z$ along $\calA$ via $t = s$.
Clearly, $\xi_x$ and $\xi_z$ agree on $\calA$.
The case where $\xi_\calA=c^q_nt^n(\frac{dt}{t})^q$ is similar.

Next, assume that $e \in E(\Gamma)$ is a bounded edge and set $y = \gtt(e)$ and $z = \gth(e)$.
The cases $y = z$ and $y \neq z$ are similar, so assume that $y\neq z$.
Consider the star-shaped local lifts $(Y, y, \xi_y)$ and $(Z,z, \xi_z)$ with open annuli $\calA_Y \subset Y$ and $\calA_Z \subset Z$.
Then the orientation of $\calA_Y$ is compatible with $e$, and of $\calA_Z$ is not.
By cutting the annuli into annuli of smaller moduli we may assume that the skeletons of $\calA_Y$ and $\calA_Z$ are short enough to fit in the edge $e$ without interesction.
Set $n=-\frac{\partial \ell}{\partial e}=\frac{\partial \ell}{\partial e^\opp}$.
By Proposition \ref{prop good coor}, there exist good analytic coordinates $t$ and $s$ on the annuli.
That is, if $n\neq 0$ and is a multiple of $q$ then we have $\xi_y |_{\calA_Y} = (\alpha t^{\frac{n}{q}} + c)^q (\frac{dt}{t})^q$
and $\xi_z|_{\calA_Z} = (\beta s^{-\frac{n}{q}} + c')^q (\frac{ds}{s})^q$ where $c^q=\gtR^q(e)$ and $c'^q=\gtR^q(e^\opp)=(-1)^q\gtR^q(e)$.
Note that $(\frac{c}{c'})^q=(-1)^q$.
Take an open annulus $\calA$ of modulus equal to the length of $e$ with coordinate $\tau$ and equip it with the $q$-differential $\xi_\calA = (\alpha \tau^{\frac{n}{q}} + c)^q (\frac{d\tau}{\tau})^q$.
Now glue $\calA$ to $\calA_Y$ and  $\calA_Z$ via $\tau = t$ and $\tau = \frac{c\beta}{c'\alpha}s^{-1}$.
The case of $n=0$ or $n$ is not a multiple of $q$ is similar and we omit its proof.

\begin{rem}
It is important to note that what makes the above gluing possible is the fact that the relevant parameters (the free terms and the exponents of the leading terms of the $q$-differentials) are determined by the given tropical reduction datum and correspond to the edge $e$ that we are working with, hence they are compatible on both ends of $e$.
\end{rem}

\section{Good coordinates} \label{sec good coor}

In order to have a better understanding of the $q$-residue function we classify $q$-differential forms without zeros and poles on an annulus $\calA$.
Furthermore, this classification will provide us with a "standard" presentation of $q$-differentials which makes glueing $q$-differentials along an annulus easy.
For this purpose we introduce the notion of a good coordinate of an annulus with respect to a $q$-differential form.

\subsection{The definition}
Recall that in \cite[\S~4.1]{TT20} the authors define a good coordinate of an oriented annulus $\calA$ with respect to a differential form $\omega$ having neither zeros nor poles on $\calA$ as a coordinate $t$ such that either $\omega=c_0\frac{dt}{t}$ or $\omega=(c_nt^n+c_0)\frac{dt}{t}$, $n\neq 0$ and $|c_nt^n|_x>|c_0|_x$ and all $x\in \calA$.
Note that $\Res_\calA(\omega)=c_0$.
We now extend the definition as follows

\begin{defin}
Let $\calA$ be an oriented annulus and $\xi$ a $q$-differential form having neither zeros nor poles on $\calA$.
A coordinate $t$  on $\calA$ will be called \emph{good} with respect to $\xi$ if either of the following cases holds
\begin{enumerate}
	\item $\xi=(c_nt^{\frac{n}{q}}+c_0)^q(\frac{dt}{t})^q$ with $q|n\neq 0$ and $|c_nt^{\frac{n}{q}}|_x>|c_0|_x$ for all $x\in \calA$.
	\item $\xi=c^q_nt^n(\frac{dt}{t})^q$.
\end{enumerate}
\end{defin}

\begin{rem}
Given a coordinate $t$ of $\calA$ we can write $\xi=f(t) (\frac{dt}{t})^q$ with $f(t)=\sum_i a_it^i\in k\{t,rt^{-1}\}$.
As $\xi$ has neither zeros nor poles the series $f(t)$ has a dominant term, i.e.\ there is some $n\in \bbZ$ such that
\[|a_it^i|_x< |a_nt^n|_x\] for all $i\neq n$ and all $x\in \calA$ (see \cite[\S~4.1.2]{BT20}).
Furthermore  if $a_nt^n$ is the dominant term of $f(t)$ then $|n|$ (the usual absolute value of $n$)  is an invariant of  $f(t)$. For another coordinate $t_1$ the dominant term of $f(t_1)$ is given by the index $n$ if the coordinate change $t_1=\phi(t)$ is orientation preserving and by the index $-n$ if the coordinate change is orientation reversing.
\end{rem}

\subsection{Good coordinate for a closed annulus} \label{sec proof good coor closed}
We now prove Proposition \ref{prop good coor} for a closed annulus.

Let $\calA$ be an analytic closed annulus and $\xi$ a $q$-differential having neither zeros nor poles on $\calA$.
We will show that $\calA$ admits a good coordinate with respect to $\xi$.

Our strategy of proof here is the same one used in the proof of Theorem \ref{theor lifting}.
That is, we pull back $\xi$ along a $q$-cover $\phi\colon \calA'\to \calA$, apply \cite[Proposition~4.1.3]{TT20} to a suitably constructed differential form $\omega$ on $\calA'$ equipped with a group action of $\mu_q$ and take the quotient to obtain the result.

\begin{proof}
Choosing a coordinate $s$ we have $\calA=\calM\left(k\{s,rs^{-1}\} \right)$.
Write $\xi=\sum_i a_is^i (\frac{ds}{s})^q$ and assume that  the dominant term is given by the index $n$.
Denoting $\lam(s)=\sum_{i\neq n} \frac{a_is^{i}}{a_ns^n}$ we have $\xi=a_ns^n\left(1+\lam(s)\right)(\frac{ds}{s})^q$.
Note that $|\lam(s)|_x<1$ for all $x\in \calA$ and $1+\lam(s)$ is a unit of $k\{s,rs^{-1}\}^\circ$.

Set $\calA'=\calM(k\{y,r^\frac{1}{q}y^{-1}\})$ and let $\phi\colon \calA'\to \calA$ be the $q$-cover given by $\phi^*(s)=y^q$.
We obtain
\[\phi^*(\xi)=a_ny^{nq}\left(1+\lam(y^q)\right) (q\frac{dy}{y})^q.\]
We can now choose a converging series $\alpha(y^q)\in k\{y^q,ry^{-q}\}^{\circ\circ} \subset k\{y,r^\frac{1}{q}y^{-1}\}^{\circ\circ}$ and $b\in k$ such that $(1+\alpha(y^q))^q=1+\lam(y^q)$ and $b^q=a_n$
and write
\[\phi^*(\xi)=\left(by^{n}\left(1+\alpha(y^q)\right)q\frac{dy}{y}\right)^q\]
or simply $\phi^*(\xi)=(q\omega)^q$ with $\omega=by^n(1+\alpha(y^q))\frac{dy}{y}$ which is a differential form on $\calA'$ having neither zeros nor poles, whose dominant term is also given by the index $n$ and unique up to multiplication with a $q$-th root of unity.

By \cite[Proposition~4.1.3]{TT20} the annulus $\calA'$ has a good coordinate $z$ with respect to $\omega$.
That is
\[\omega=
	\begin{cases}
	c_0\frac{dz}{z}& \text{if } n=0 \\
	(c_nz^n+c_0)\frac{dz}{z} & \text{if } n\neq0 .
	\end{cases}\]
As the different choices of $\omega$ differ only by multiplication with a constant, we see that $z$ is already a good coordinate with respect to all choices of $\omega$.

Let $\mu_q=\left< g \right>$ be the cyclic group of order $q$ and let $\mu_q$ act on $\calA'$ by $g(y)=\zeta y$ where $\zeta=\zeta_q$ is a primitive $q$-th root of unity.
Clearly $\phi\circ g=\phi$.
It is also clear that $g\left(\alpha(y^q)\right)=\alpha(g(y)^q)=\alpha(y^q)$.
It follows that
\[\omega=by^n(1+\alpha(y^q))\frac{dy}{y}=\zeta^{-n}bg(y)^n(1+\alpha(g(y)^q))\frac{dg(y)}{g(y)}\]
or in other words
\[g(\omega)=\zeta^n\omega=\zeta^j\omega\]
for $0\le j\le q-1$ such that $j \equiv n \mod q$.

Given a differential form $\eta$ and a coordinate $y=y_0$, \cite[Proposition~4.1.3]{TT20} constructs a good coordinate $z$ with respect to $\eta$ as the limit of successive approximations $y_j$ each obtained from the previous by multiplication by a unit $u_j$.
Furthermore, $u_j$ is of the form $u_j=1+v_j$ where $|v_j|_x<1$ for all $x\in \calA'$ and $v_j$ is constructed from the coefficients of the presentation of $\eta$ with respect to the coordinate $y_{j-1}$.
More precisely, given $\eta=\displaystyle\sum_if_iy^i\frac{dy}{y}$ with dominant term $f_ny^n$ for $n\neq 0$, the coordinate change $y_1=yu_1$ is given by choosing $u_1$ to be the unique $n$-th root of $1+\displaystyle\sum_{i\neq 0,n} \frac{n}{i}\frac{f_iy^i}{f_ny^n}$ satisfying $|u_1-1|_x<1$ for all $x\in \calA'$ (in other words, $u_1$ is the unique $n$-th root of $1+\displaystyle\sum_{i\neq 0,n} \frac{n}{i}\frac{f_iy^i}{f_ny^n}$ whose free coefficient is $1$).
Writing $\eta=f_ny^n(1+\displaystyle\sum_{l\neq 0}h_ly^l)$ with $h_l=\frac{f_{l+n}}{f_n}$, we have
\[u_1^n=1+\sum_{l\neq 0,-n}\frac{n}{l+n}h_ly^l.\]
Applying the above to $\omega=by^n(1+\alpha(y^q))\frac{dy}{y}$ with $\alpha(y^q)=\displaystyle\sum_{i\neq 0} \alpha_{iq}y^{iq}$ we obtain
\[u_1^n=1+\sum_{i\neq 0,-\frac{n}{q}}\frac{n}{iq+n}\alpha_{iq}y^{iq}.\]
It follows that $g(u_1)=u_1$ hence $g(y_1)=g(y)u_1=\zeta y_1$.
Repeating the argument for all $j$ we have that $u=\prod u_j$ is also $\mu_q$-invariant so the good coordinate $z=yu$ is $\mu_q$-invariant as well.
A quick check shows that the same conclusion also applies to the $n=0$ case.
We conclude that $t=z^q$ defines a coordinate on $\calA$.

Now, $\omega$ is $\mu_q$-invariant whenever $q$ divides $n$ (including $n=0$) as $g(\omega)=\zeta^n\omega$.
When $q$ does not divide $n$ we have
\[\zeta^n(c_nz^n+c_0)\frac{dz}{z} =\zeta^n\omega=g(\omega)=(c_n\zeta^nz^n+c_0)\frac{dz}{z}\]
so $c_0$ must be zero.
That is $t$ is a good coordinate with respect to $\xi$.
\end{proof}

\begin{rem}
When $n=0$ the coordinate $t$ can also be found in a more direct way, maybe even explicitly.
We are looking for a unit $u(s)$ such that $\xi=a_0(\frac{dt}{t})^q$ for the coordinate $t=su$. Set $v(s)=\displaystyle\sum_{i\neq 0} \frac{a_i}{a_0}s^i$.
Note that $v(s)$ is an element of $k\{s,rs^{-1}\}$ and $|v|_x<1$ for all $x\in \calA$ as $0$ is the index of the dominant term of $\xi$.
Writing $\frac{dt}{t}=(1+\frac{s}{u}\frac{du}{ds})\frac{ds}{s}$ we obtain the equation
\[\left( 1+\frac{s}{u}\frac{du}{ds} \right)^q=1+v.\]
Recalling the Taylor expansion of $(1+x)^\frac{1}{q}$ we obtain the equation
\[\frac{du}{u}=\sum_{m=1}^\infty \binom{\frac{1}{q}}{m}v^m\frac{ds}{s}\]
which is solvable.
\end{rem}

\subsection{Good coordinate of an open annulus}
We now extend the result to open annuli.

Let $\calA=\calA(r,1)$ be an oriented open annulus of modulus $0 < r < 1$ and $\xi$ a $q$-differential form having neither zeros nor poles on $\calA$.
Let $s$ be a coordinate of $\calA$.
Choosing sequences $(r_i)$ and $(r'_i)$ such that $\displaystyle\lim_{i\to \infty} r_i=r$ , $\displaystyle\lim_{i\to \infty} r'_i=1$ and $r_i<r'_i$ for all $i$ we form the covering $\calA=\cup _i\calA[r_i,r'_i]$.
The restriction $\xi|_{\calA[r_i,r'_i]}$ is a $q$-differential form having neither zeros nor poles on $\calA[r_i,r'_i]$ and $s_i=s|_{\calA[r_i,r'_i]}$ is a coordinate of $\calA[r_i,r'_i]$.
As we saw above $\calA[r_i,r'_i]$ has a good coordinate $z_i$ with respect to $\xi|_{\calA[r_i,r'_i]}$.
Furthermore, by applying the process detailed in \S\ref{sec proof good coor closed} to $s_j$ for some $j$, we see that not only is the obtained $z_j$ a good coordinate of $\calA[r_j,r'_j]$ with respect to $\xi|_{\calA[r_j,r'_j]}$, but also for any $i<j$ we have that $z_j|_{\calA[r_i,r'_i]}$ is a good coordinate of $\calA[r_i,r'_i]$ respect to $\xi|_{\calA[r_i,r'_i]}$.

A similar argument applies to semi-open annuli of the form $\calA[r,1)$ and $\calA(r,1]$ and in particular to a \emph{punctured disc} i.e.\ an analytic space isomorphic to $D\setminus\{O\}$ where $D=\calM(k\{t\})$ is the closed unit disc and $O\in D$ is the origin i.e.\ the point of type 1 corresponding to $0\in k$.
Note however that the induced orientation of an annulus is the one in which the coordinate is decreasing along the skeleton while the induced orientation of a punctured disc $D\setminus\{O\}$ is the one in which the coordinate is increasing along the skeleton $(O,1]$.

\subsection{Orientation and $q$-residue}
The results of the previous sections are invariant under orientation preserving coordinate change. In other words, given a $q$-differential form $\xi$ on $\calA=\calM( k\{s,rs^{-1}\})$ then the coeficient $c_0^q$ found above is well defined up to an orientation preserving coordinate change.
To complete the picture to orientation reversing coordinate changes as well we only need to consider the coordinate change $t=s^{-1}$ of a good coordinate $s$.
Assuming $\xi=(c_ns^{\frac{n}{q}}+c_0)^q(\frac{ds}{s})^q$ we have
\[\xi=(c_nt^{-\frac{n}{q}}+c_0)^q(-\frac{dt}{t})^q=(-1)^{q}(c_nt^{\frac{-n}{q}}+c_0)^q(\frac{dt}{t})^q \]
and we see that $t$ is also a good coordinate as the index of the dominant term changes sign under an orientation reversing coordinate change.
We conclude that under an orientation reversing coordinate change, the free coefficient with respect to a good coordinate is multiplied by $(-1)^q$ as we already saw in Proposition \ref{prop_properties of residue function}(1).

\bibliographystyle{amsalpha}

\bibliography{q-diff}

\newcommand{\etalchar}[1]{$^{#1}$}
\providecommand{\bysame}{\leavevmode\hbox to3em{\hrulefill}\thinspace}
\providecommand{\MR}{\relax\ifhmode\unskip\space\fi MR }
\providecommand{\MRhref}[2]{%
  \href{http://www.ams.org/mathscinet-getitem?mr=#1}{#2}
}
\providecommand{\href}[2]{#2}
\begin{thebibliography}{BCG{\etalchar{+}}18}

\bibitem[BCG{\etalchar{+}}16]{BCGGM16}
Matt Bainbridge, Dawei Chen, Quentin Gendron, Samuel Grushevsky, and Martin
  M{\"o}ller, \emph{Strata of $ k $-differentials}, arXiv preprint
  arXiv:1610.09238 (2016).

\bibitem[BCG{\etalchar{+}}18]{BCGGM18}
\bysame, \emph{Compactification of strata of abelian differentials}, Duke
  Mathematical Journal \textbf{167} (2018), no.~12, 2347--2416.

\bibitem[BT20]{BT20}
Uri Brezner and Michael Temkin, \emph{Lifting problem for minimally wild covers
  of berkovich curves}, Journal of Algebraic Geometry \textbf{29} (2020),
  no.~1, 123--166.

\bibitem[EV92]{EV92}
H\'{e}l\`{e}ne Esnault and Eckart Vieweg, \emph{Lectures on vanishing
  theorems}, Mathematical Surveys and Monographs, vol.~20, Birkhäuser Basel,
  1992.

\bibitem[GT17]{GT}
Quentin Gendron and Guillaume Tahar,
  \emph{Diff$\backslash$'erentielles$\backslash$a singularit$\backslash$'es
  prescrites}, arXiv preprint arXiv:1705.03240 (2017).

\bibitem[MOL18]{M18}
MARTIN MOLLER, \emph{Geometry of teichm{\"u}ller curves}, Proceedings of the
  International Congress of Mathematicians (ICM 2018) (In 4 Volumes)
  Proceedings of the International Congress of Mathematicians 2018, World
  Scientific, 2018, pp.~2017--2034.

\bibitem[Sch16]{S}
Johannes Schmitt, \emph{Dimension theory of the moduli space of twisted $ k
  $-differentials}, arXiv preprint arXiv:1607.08429 (2016).

\bibitem[Tem16]{Metrization}
Michael Temkin, \emph{Metrization of differential pluriforms on berkovich
  analytic spaces}, pp.~195--285, 01 2016.

\bibitem[TT20]{TT20}
Michael Temkin and Ilya Tyomkin, \emph{Reduction and lifting problem for
  differential forms on berkovich curves}, arXiv preprint arXiv:2005.01397
  (2020).

\bibitem[Wri15]{W}
Alex Wright, \emph{Translation surfaces and their orbit closures: an
  introduction for a broad audience}, EMS Surveys in Mathematical Sciences
  \textbf{2} (2015), no.~1, 63--108.

\end{thebibliography}

\end{document}